\documentclass[12pt]{amsart}
\usepackage[margin=3cm]{geometry}
\usepackage{amssymb}
\usepackage{latexsym}
\usepackage[all,cmtip,arrow,2cell]{xy}
\usepackage{graphicx}
\usepackage{hyperref}

\theoremstyle{plain}
\newtheorem{thm}{Theorem}[section]
\newtheorem{prop}[thm]{Proposition} 
\newtheorem{lemma}[thm]{Lemma}
 
\newtheorem{dfn}[thm]{Definition}

\theoremstyle{definition}
\newtheorem{rmk}[thm]{Remark}
\newtheorem{ex}[thm]{Example} 
\newtheorem{notation}[thm]{Notation}

\newcommand{\mapsfrom}{\reflectbox{$\mapsto$}}
\newcommand{\cinf}{C^{\infty}}
\newcommand{\tot}{\rm{tot}}
\newcommand{\en}{\rm{en}}
\newcommand{\mom}{\rm{mo}}

\newcommand{\reals}{{\mathbb R}}

\newcommand{\integers}{{\mathbb Z}}

\newcommand{\vol}{{\rm vol}} 

\newcommand{\gradg}{\grad _g}
\newcommand{\gradgam}{\grad _\gamma}

\newcommand{\cald}{{\mathcal D}} \newcommand{\cale}{{\mathcal E}}
\newcommand{\calf}{{\mathcal F}} \newcommand{\cali}{\mathcal{I}} 
 \newcommand{\calh}{{\mathcal H}}
 \newcommand{\call}{{\mathcal L}}
\newcommand{\calm}{{\mathcal M}}  

\newcommand{\calp}{{\mathcal P}}

\newcommand{\cals}{{\mathcal S}}
 \newcommand{\calv}{{\mathcal V}}
\newcommand{\calu}{{\mathcal U}} 
\newcommand{\calx}{{\mathcal X}}

\newcommand{\arrows}{\,\lower1pt\hbox{$\longrightarrow$}\hskip-.24in\raise2pt
             \hbox{$\longrightarrow$}\,}
\newcommand{\defequal}{\stackrel{\mbox {\tiny {def}}}{=}}

\newcommand{\from}{\leftarrow}

\newcommand{\calF}{\mathcal{F}}

\newcommand{\calL}{\mathcal{L}} 
\newcommand{\calU}{\mathcal{U}} 
\newcommand{\calX}{\mathcal{X}}
\newcommand{\bbR}{\mathbb{R}} 
 
\newcommand{\Lie}{\mathcal{L}}

\DeclareMathOperator{\Tr}{Tr}

\DeclareMathOperator{\divergence}{div}
\DeclareMathOperator{\grad}{grad}
\DeclareMathOperator{\Ric}{Ric}

\begin{document}

\title[Groupoid symmetry and constraints in general relativity]
{Groupoid symmetry and constraints\\ in general 
relativity} 

\author[C.~Blohmann]{Christian Blohmann}
\address{Max-Planck-Institut f\"ur Mathematik, Vivatsgasse 7, 53111 Bonn, Germany}
\email{blohmann@mpim-bonn.mpg.de}

\author[M.C.B.~Fernandes]{Marco Cezar Barbosa Fernandes}
\address{Instituto de F\'isica, Universidade de Bras\'ilia,
70910-900, Bras\'ilia, DF, Brasil}
\email{mcezar@fis.unb.br}

\author[A.~Weinstein]{Alan Weinstein}
\address{Department of Mathematics, University of California, Berkeley, CA 94720, USA.}
\email{alanw@math.berkeley.edu}

\thanks{Partial support for Christian Blohmann was provided by a Marie Curie
    Fellowship from the European Union, for Marco Fernandes 
  by a CAPES grant from the Government of Brazil, and for
Alan Weinstein by grants
  DMS-0204100 and DMS-0707137 from the US National Science Foundation.}

\subjclass[2010] {83C05 (Primary); 58H05 (Secondary)}


\keywords{Einstein equations, constraints, Poisson brackets, 
groupoid, Lie algebroid, diffeology}

\begin{abstract}
  When the vacuum Einstein equations are cast in the form of hamiltonian
  evolution equations, the initial data lie in the cotangent bundle
  of the manifold $\calm\Sigma$ of riemannian metrics on a Cauchy hypersurface
  $\Sigma$.   As in
   every lagrangian field theory with symmetries, the
  initial data must satisfy constraints. But, unlike those of gauge
  theories, the constraints of general relativity do not arise as
  momenta of any hamiltonian group action.  In this paper, we show
  that the bracket relations among the 
constraints of general relativity are identical to the bracket
relations in the Lie algebroid of a
groupoid consisting of diffeomorphisms between space-like
  hypersurfaces in spacetimes.   

A direct connection is still missing between the constraints themselves, whose
definition is closely related to the Einstein equations, and our
groupoid, in which the Einstein 
equations play no role at all.   We discuss some of the difficulties
involved in making such a connection.
\end{abstract} 

\dedicatory{To Darryl Holm, for his 64th  birthday}

\maketitle

\section{Introduction}
\label{sec-intro}
The vacuum Einstein equations
 state that
the Ricci curvature $\Ric (g)$ of a lorentzian metric $g$ 
is identically zero.  Recast
as evolution equations, they become a hamiltonian
system on the cotangent bundle of the manifold\footnote{In this paper,
  we will actually
 treat 
spaces of smooth functions 
as diffeological spaces rather than as
 Fr\'echet manifolds.  In Appendix A, we review the theory of
diffeology, concentrating on the aspects
which are relevant to our work.  
}  ${\mathcal M}\Sigma$
of smooth riemannian metrics on a manifold
 $\Sigma$ which represents
the typical Cauchy hypersurface\footnote{In general
  relativity, $\Sigma$ has dimension 3, 
but that assumption is not necessary for anything
  we do in this paper.}.  Each element of  $T^*{\mathcal M}\Sigma$
may be identified with a pair $(\gamma,\pi)$, where $\gamma$ is a
riemannian metric and $\pi$ is a symmetric
covariant 2-tensor field.\footnote{
Strictly speaking, a cotangent vector $\pi$ to ${\mathcal M}\Sigma$
should be a contravariant symmetric 2-tensor density, but 
we may use the metric $\gamma$ and its associated volume element
 to identify covariant and contravariant tensors, and to
identify scalar functions with densities.  In addition, we 
consider $\pi$ as an endomorphism of the tangent bundle $T\Sigma$ when
we form $\pi^2$ and take traces.}

It has long been known\footnote{See
 Section \ref{subsec-history} below for historical remarks and references.} 
that, for a given  $\gamma$ and $\pi$ to be admissible as initial conditions for
the Einstein equations, they must satisfy a system of constraint equations.
These equations  may be derived either geometrically 
from the Gauss-Codazzi equations relating the intrinsic and extrinsic
curvatures of a hypersurface, or from a lagrangian formulation of the
Einstein equations in terms of the
Einstein-Hilbert action.  The equations are:
\begin{align}
{\mathbf C}_{\mom}(\gamma,\pi) &:= -2 \divergence_\gamma
\pi = 0\,,\qquad\\ 
{\mathbf C}_{\en}(\gamma,\pi)
&:=  -R(\gamma) + \Tr_\gamma (\pi^2) - \tfrac{1}{\dim \Sigma - 1} (\Tr_\gamma \pi)^2 = 0\,.
\end{align}
The {\em momentum
  constraint}, ${\mathbf C}_{\mom}$,
maps $T^*{\mathcal M}\Sigma$ to the
space $\calx{\Sigma}$ of vector fields on $\Sigma$, while the
{\em energy constraint}, ${\mathbf
  C}_{\en}$,  takes values in the space $\calf{\Sigma}$ of 
scalar functions on $\Sigma$.
The constraints may be viewed as those components of the Einstein tensor
which involve directions normal to the Cauchy hypersurface.  (See Appendix E of
\cite{Wald} for details.)  

The {\em constraint set} ${\mathcal C} \subset 
T^*{\mathcal M}\Sigma$, where the constraints
are all equal to zero,
has two properties which always
hold for the zero sets of  momentum\footnote{Note that we are
  using two meanings of ``momentum'' in this discussion, first in a
  slightly extended version of the usual ``mass times velocity'', and
  second in the sense used in the theory of hamiltonian
  actions.   In the latter sense, the term ``moment'' is often used
  instead.  (See the footnote on p.~133 of \cite{mi-we:moments}) for
  some remarks on the two nomenclatures.)}
maps of proper hamiltonian group actions.

\begin{itemize}
\item
The constraint set is coisotropic; i.e., for any two
functions vanishing on ${\mathcal C}$, their Poisson bracket vanishes
there as well.  (This follows from the bracket formulas \eqref{brackets}
below.)  Consequently, on its regular part,
namely those pairs $(\gamma,\pi)$ with no common infinitesimal
symmetries, ${\mathcal C}$ is foliated by {\em characteristic
  submanifolds} whose tangent spaces are spanned by the hamiltonian
vector fields whose hamiltonians vanish on ${\mathcal C}$.   In the
group action setting, these would be the orbits of the symmetry group.   

\item
As shown by Arms, Marsden, and Moncrief
\cite{ar-ma-mo:structure}, the constraint set has quadratic
singularities at the points 
which do admit infinitesimal symmetries.  In the group action setting,
this would follow from the linearizability of proper actions and the equivariant
Darboux theorem.  
\end{itemize}

The aim of the research described in this paper has been to identify
the symmetry structure responsible for the constraints and their
Poisson bracket relations.
What we have found is that the bracket relations, rather than coming from the
Lie algebra of a symmetry group, 
are those of 
a Lie algebr{\em oid} which is derived from a group{\em
  oid}\footnote{The definition of groupoid is reviewed briefly in
  Section \ref{sec-diffgroupoid}.  We refer to \cite{ma:general} for a
  full treatment of Lie algebroids and Lie groupoids.}  of
diffeomorphisms between space-like hypersurfaces in Lorentz manifolds.
Unfortunately, this groupoid, which encodes
 the arbitrariness in the choice of
initial value hypersurface for the Einstein equations as well as of
coordinates on this hypersurface, lives over a much larger space than the 
one where the constraints are defined.  It remains to be seen what the
relevant structure is which connects our groupoid with the constraints
themselves.  

Since the constraints ${\mathbf C}_{\mom}$ and ${\mathbf C}_{\en}$ are
vector valued, we break them into scalar-valued components to form
their Poisson brackets.
Following DeWitt
 \cite{dw:QuantGrav}, we 
get these components by pairing the constraints by
integration against vector fields and
functions on $\Sigma$, obtaining for each vector field $X$ and
function $\phi$ the following
real-valued constraint function on  $T^*{\mathcal M}\Sigma$: 
\begin{equation}
\label{eq:Constraint1}
  C_{(X,\phi)} (\gamma,\pi)
  = \int_\Sigma \bigl\{
   \gamma(X, {\mathbf C}_{\mom}(\gamma,\pi))  
   + \phi {\mathbf C}_{\en}(\gamma,\pi) 
  \bigr\} \vol_\gamma \,.
\end{equation}
Since $C_{(X,\phi)}$ is the sum of 
 $C_X\defequal C_{(X,0)}$ and $C_\phi  \defequal C_{(0,\phi)}$, it suffices
to write down the Poisson bracket relations among these terms.  These
were found by
DeWitt 
to be:
\begin{align} \label{brackets}
\{C_X,C_Y\} &=  C_{[X,Y]}   \notag \\
\{C_X,C_\phi\} &=  C_{X\cdot\phi}   \\
\{C_\phi,C_\psi\} &=  C_{\phi\,\gradgam\psi - \psi\,\gradgam\phi} \notag  = 
C_{\gamma{^{-1}}(\phi d \psi - \psi d  \phi)} \,,
\end{align}
where the metric $\gamma$ is here considered as a bundle map from
$T\Sigma$ to $T^*\Sigma$, so that its inverse takes the differentials
of functions to their gradients.

The coisotropic property of the constraint set follows immediately
from \eqref{brackets} above:
the bracket of any two constraint functions vanishes on the
constraint set.   On the other hand, 
the dependence of $\{C_\phi,C_\psi\}$  on the metric $\gamma$
means that the brackets are not those of a
fixed Lie algebra structure on $\calX\Sigma \oplus \calF\Sigma$.
Of course, we may freeze the metric $\gamma$ to
some value $\overline\gamma$,
but then the 
resulting bracket $\{~,~\}_{\overline\gamma}$ will
 not satisfy the Jacobi identity.  The
anomaly appears in the jacobiator of a momentum constraint and two
energy constraints, namely:
\begin{equation*}
\{C_X,\{C_\phi,C_\psi\}_{\overline\gamma}\}_{\overline\gamma} + {\rm circ. perm.} 
= C_{\mathcal L_X (\overline \gamma^{-1})(\phi d \psi - \psi d \phi)} \,,
\end{equation*}
which vanishes only in special cases, such as when $X$ is a
Killing vector field for $\overline\gamma$.

We must therefore renounce the idea that the constraints should be the
momentum map of a symmetry group (such as the group of diffeomorphisms
of spacetime).  Instead, we use a groupoid.

The objects of our groupoid will be the isometry classes of embeddings of
$\Sigma$ as a cooriented, space-like hypersurface in a
lorentzian manifold.  We call these $\Sigma$-universes. 
The morphisms of the groupoid, which we call $\Sigma$-evolutions,
 will be the isometry classes of 
\emph{pairs} of such embeddings into the \emph{same}
 target lorentzian manifold; they can
be identified with diffeomorphisms between the image hypersurfaces.  
Remarkably, the Lie
algebroid of this groupoid is naturally a trivial bundle, and the 
bracket relations among its constant sections turn out to reproduce precisely
the bracket relations 
\eqref{brackets} among the constraints.  

To facilitate computations with the rather abstractly defined spaces of
isometry classes,  
we will use the fact that each $\Sigma$-universe has near the image of
$\Sigma$ a
unique gaussian representative, i.e. a metric near (in the appropriate
sense) $\Sigma\times \{0\}$ on $\Sigma\times\reals$ for which the
paths $t\mapsto (x,t)$ for each $x\in \Sigma$ are time-like geodesics
parametrized by (negative) arc-length and normal to $\Sigma\times \{0\}$.

As an aside, we will show that the
groupoid of $\Sigma$-evolutions is equivalent in a precise sense to
the groupoid of isometries between those Lorentz manifolds which admit a
cooriented, space-like hypersurface diffeomorphic to $\Sigma$.

There remains a significant gap between our results and 
a satisfactory  geometric explanation of the relations
\eqref{brackets} among the constraints, since our Lie algebroid lies
over a space much bigger than the phase space  $T^*\calm\Sigma$ on
which the constraints live.  The latter space may be identified, via a
natural riemannian 
metric on $\calm\Sigma$, with the tangent bundle $T\calm\Sigma$, and
this tangent bundle may be in turn identified with the space of 
1-jets
around $\Sigma$ of $\Sigma$-universes.  
Differentiation by tangent vectors based along $\Sigma$ gives a
natural projection from the $\Sigma$-universes to the 1-jets.  
Unfortunately, it does not seem to be possible to make our Lie
algebroid descend along this projection.

To end this introduction, we note that Teitelboim \cite{Teitelboim} already
gave an argument leading to the Poisson bracket relations among the constraints
using pure geometry, without any appeal to Einstein's equations.   In some sense,
the accomplishment of our paper is to put Teitelboim's argument in its proper
mathematical context, that of groupoids and Lie algebroids.

\section*{Acknowledgements}

For comments, encouragement, and advice, we would like to thank 
John Baez, Robert Bryant, Uli Bunke, Arthur Fischer, Darryl Holm, Patrick Iglesias-Zemmour, 
 Klaas Landsman, Jerry Marsden, Peter
Michor, and Phil Morrison,  as well as the audiences over several years 
who heard us present preliminary (and sometimes overoptimistic)
versions of this work, beginning with the 2007 Lausanne conference 
on the occasion of Darryl's 60th birthday.

Much of this work was done when we were away from home.  We would like
to extend our appreciation for their hospitality
to the Department of Mathematics at UC Berkeley (Blohmann and
Fernandes), 
the Freiburg Institute of Advanced Studies and the Mathematische
Forschungsinstitut Oberwolfach
 (Blohmann), and the \'Equipe d'Analyse
Alg\'ebrique at the Institut Math\'ematique de Jussieu (Weinstein).

\section{Universes}
We turn now to the construction of a groupoid whose Lie algebroid
is a trivial bundle with fibre
$\calX\Sigma \oplus \calF\Sigma$, such that
the bracket of constant
sections is given by (\ref{brackets}).  
\subsection{$\Sigma$-universes}

A connected lorentzian manifold $M$ will be called a {\bf
  spacetime}\footnote{Note that any such manifold satisfies an
  Einstein equation of the form $\Ric (g) = T$ if the energy-momentum
  tensor $T$ is simply defined by that equation.}  
(whether or not its dimension is 4).  
$M$ will be called $\Sigma$-{\bf adapted}
if it admits a cooriented (i.e. with an orientation of the normal
bundle), 
proper embedding of $\Sigma$ as a
space-like hypersurface; such an embedding will be called a
$\Sigma$-{\bf space} in $M$, and a pair consisting of a spacetime and
a $\Sigma$-space in it will be called a $\Sigma$-{\bf spacetime.}

Since the identity of specific points in the ambient spacetime 
is irrelevant\footnote{Dirac
  \cite{di:GravHam} wrote, ``I am 
  inclined to believe \ldots that four-dimensional symmetry is not a
  fundamental property of the physical world.''   Pirani
  \cite{pi:diracreview},  reviewing Dirac's paper, ``finds it
  difficult to concur''.}  
to the evolution of metrics on a manifold
 $\Sigma$, it is natural to make Definition \ref{df:universe} below.

\begin{notation}
For better compatibility with composition, we will often represent mappings
by arrows going from right to left; hence we write $\cinf(X,Y)$ 
for the smooth
mappings \emph{to} $X$ \emph{from} $Y$.
\end{notation}

\begin{dfn}
\label{df:universe}  
A $\Sigma$-{\bf universe} is an equivalence class of $\Sigma$-spacetimes, 
where $M \stackrel{i}{\hookleftarrow} \Sigma$ 
and  $ M' \stackrel{i'}{\hookleftarrow} \Sigma$  are 
equivalent if there is an isometry $M' \stackrel{\psi}{\leftarrow} M$ which is
consistent with the coorientations and which satisfies $\psi \circ i =
i'$ .  We will denote the set\footnote{The collection of all
  $\Sigma$-spacetimes
 is, like the collection of all sets,
  a ``class'' in the set-theoretic sense rather than a set.
  But the $\Sigma$-universes $\calU\Sigma$ do form a set
  because every connected manifold is diffeomorphic to a submanifold
  of some $\reals^n$.}  of all $\Sigma$-universes by $\calU\Sigma$.
\end{dfn}

To define a diffeology on 
 $\calU\Sigma$, since different $\Sigma$-universes may be represented
 by different lorentzian manifolds, we follow 
the pattern described in Appendix A for functional
 diffeologies on spaces of mappings with varying domains.
Namely, we stipulate that 
a parametrization $\calU\Sigma\stackrel{\phi}{\from} P$ is 
smooth
if: 
\begin{enumerate}
\item each point of $P$ has a neighborhood $\calu$ for which there is
  a fixed manifold $M$ such that the
  $\Sigma$-universes $\phi (p)$,
$p\in \calu$, are
represented by $\Sigma$-spacetimes
$M_p\stackrel{i_p}{\hookleftarrow} \Sigma,$ where all the 
$M_p$ are open subsets of $M$;

\item $i_p\mapsfrom p$ is a smooth map 
 $C^\infty(M,\Sigma) \leftarrow P$.

\item $M_{\tot}=\{(x,p)| p\in P~{\rm and~} x\in M_p\}$ 
is open in $M\times P$; 

\item
the lorentzian structures on the fibres are the
restrictions of a smooth section over $M_{\tot}$ of the bundle of
symmetric fibrewise 2-forms; 
\end{enumerate}

The following proposition gives a first description of the tangent
bundle of $\calU\Sigma$.  A more explicit description will follow the
introduction of gaussian splittings.

\begin{prop}
\label{prop:tangentuniverses}
The tangent cone to $\calU\Sigma$ at a $\Sigma$-universe
$[(M,g)\stackrel{i}{\hookleftarrow} \Sigma]$
is a vector space which may be
identified with the quotient of the space $\Gamma(S^2(T^*M))$
of symmetric 2-forms on $M$
by the image of the  map $\Gamma(S^2(T^*M)) 
\stackrel{\calL g}{\from}\calX_{i(\Sigma)} M$ 
taking each vector field $Z$ defined on $M$ and 
vanishing on $i(\Sigma)$ to the Lie derivative $\calL_Z g$.   
\end{prop}

\begin{proof}
A tangent vector to $\calu\Sigma$ is represented by a 1-parameter
family $(M_s,g_s)\stackrel{i_s}{\hookleftarrow}\Sigma$ of embeddings, where 
$M_0=M$, $M_s$ are open subsets of
a fixed manifold $\overline{M}$, and $g_s$ are lorentzian metrics depending
smoothly upon $s$.   

We will simplify the representative by fixing  $M_s$ and the embeddings,
so that it is only $g$ which varies.  To do this, let $\xi_s$ be
the vector field  $di_s/ds=\xi_s$
along $i_s$.  Since $i_s$ is a proper embedding for each $s$,
we may ``extend'' the family $\xi_s$ over $\overline{M}$, i.e. we may choose a
smooth family $\sigma_s$ of vector fields on $\overline{M}$ such that
$\xi_s=\sigma_s\circ i_s$ for all $s$.  This family may not be
complete, but we may integrate it as far as is possible.  The result
is a family of open sets $M_s' \subset \overline{M}$ and a family of open embeddings
$M_s\stackrel{\psi_s}{\from}M_s'$ defining a smooth path in
$\cinf(\overline{M},[\overline{M}])_{\rm{open}}$ such that $M_0' =
M_0$, $\psi_0$ is the identity 
on $M_0$, and $d\psi_s/ds = \sigma_s \circ \psi_s$.  Since
$\xi_s=\sigma_s\circ i_s$, 
$\psi_s\circ i_0 = i_s$.  (In particular, $M_s'$ contains
$i_0(\Sigma)$ for all $s$.)
We obtain a new family 
 $(M_s',g_s')\stackrel{i_s'}{\hookleftarrow} \Sigma$
of embeddings by setting $g_s' = \psi_s^*(g_s)$ and $i_s'=i_s \circ
\psi_s^{-1}= i_0$.  This family is not quite equivalent to the
original one, since the domains $M_s'$ are not mapped by $\psi_s$ to
all of $M_s$, but they do agree at $s=0$.  According to the observation
in Appendix A that 
tangent
vectors are insensitive to variations of domain, this insures that the
paths in $\calu\Sigma$
represented by $(M_s,g_s)\stackrel{i_s}{\hookleftarrow} \Sigma$  and
$(M_s',g_s')\stackrel{i_s'}{\hookleftarrow} \Sigma$ represent the same tangent vector
to $\calu\Sigma$ at $[(M,g)\stackrel{i_0}{\hookleftarrow} \Sigma]$,
where the latter path has constant
embedding $i_0$.  We may keep $M_s'$ constant as well, in which case the only
object varying with $s$ is the lorentzian metric $g_s$.   Its
derivative with respect to $s$ at $s=0$ is the required smooth section
of $S^2(T^*M)$.  

Our representative of the tangent vector is not unique, since we may
apply an arbitrary family of diffeomorphisms of $M$ which fix $i(\Sigma)$ and
the normal orientation there.  (This freedom corresponds, essentially,
to the freedom which we had in extending  $\xi_s$ to 
$\sigma_s$.)   The only such diffeomorphism which can preserve a
lorentzian metric on (connected) $M$ is the identity.  The
infinitesimal action of these diffeomorphisms, i.e. the map
$\calL _\xi g \mapsfrom \xi $, 
is therefore injective, and the tangent space to $\calu\Sigma$ at
$[i]$ is the quotient of $S^2(T^*M)$ by the image of that map.  
\end{proof}

\subsection{$\Sigma$-evolutions}
\label{sec-evolutions}

The relative positions of
pairs of $\Sigma$-spaces in the same spacetime, up to equivalence,
 form a groupoid over $\calu\Sigma$ which will be our fundamental
 symmetry structure.

\begin{dfn}
A $\Sigma$-{\bf evolution} is an equivalence class of pairs 
$(i_1,i_0)$ of $\Sigma$-spaces in the same
spacetime, where a pair 
$(i_1,i_0)$ in $M$
is equivalent to $(i_1 ',i_0 ')$ in $M'$ if there is
a single isometry $M' \stackrel{\psi}{\leftarrow} M $ which is consistent 
with the coorientations and which satisfies both
$\psi \circ i_1 = i_1'$ and $\psi \circ i_0 = i_0'$.  
We will denote the set of all
$\Sigma$-evolutions by $\cale \Sigma$.  
\end{dfn}

The $\Sigma$-evolutions 
form a diffeological groupoid over the $\Sigma$-universes
with the diffeology on $\cale\Sigma$, like that on $\calu\Sigma$,
defined in terms of representatives.  The groupoid structure has as
target and source
the projections $[i_1,i_0 ] \stackrel{l}{\leftarrow} [i_1]$ 
and $[i_1,i_0 ] \stackrel{r}{\leftarrow} [i_0]$ 
(square brackets denoting equivalence classes); the composition law
is $[i_2,i_1 ][i_1,i_0 ]=[i_2,i_0 ]$; and the inversion rule
$[i_1,i_0]^{-1} = [i_0,i_1]$.

We will show in Section \ref{subsec-stack}
that $\cale\Sigma$ is Morita equivalent to the isometry groupoid
$\cali\Sigma$ of $\Sigma$-adapted spacetimes. This implies that the
orbits of $\cale\Sigma$ are in bijection with the orbits of
$\cali\Sigma$, which are the isometry classes of $\Sigma$-adapted
spacetimes.

The isotropy group of $[i]$ consists of all pairs
$([i_1,i_0])$ such that $[i_1]=[i_0]=[i]$.  For such a pair,
$M \stackrel {i_1}{\hookleftarrow}\Sigma$
and 
$M \stackrel {i_0}{\hookleftarrow}\Sigma$
are
equivalent, which means that there is an isometry $\psi$ from $M$ to itself
such that $i_1=\psi \circ i_0$.  Such an isometry is unique, if it
exists, which implies that the isotropy group of $[i]$ is isomorphic
to the isometry group of its target spacetime.

\begin{rmk}
Each $(i_1,i_0)$ corresponds to a diffeomorphism $i_1 \circ
i_0^{-1}$ between $\Sigma$-spaces in the same spacetime, and
composition in $\cale\Sigma$ corresponds to composition of
diffeomorphisms.  
\end{rmk}

Elements of the Lie
algebroid of $\cale\Sigma$ are infinitesimal $\Sigma$-evolutions and may be
parametrized by triples consisting of $\Sigma$-universes, ``shift'' vector
fields on $\Sigma$, and ``lapse'' functions on $\Sigma$, per the
following proposition.

\begin{prop}
The Lie algebroid $A\cale\Sigma\rightarrow\calu\Sigma$ 
is isomorphic as a vector bundle to the
trivial bundle $\calu\Sigma \times (\calX\Sigma \oplus \calF\Sigma)$.
\end{prop}

\begin{proof}
An element of $A\cale\Sigma$ at the base point
$[i_0] \in \calu\Sigma$ is a tangent vector to a smooth
path in $\cale\Sigma$ whose image is contained 
in the $r$-fibre of $[i_0]$ starting
at the unit $[i_0,i_0]$.   Such a path is represented by
 a smooth family $(i_s,i_0)$
defined on an interval containing $s=0$, and its tangent
vector at $s=0$ corresponds to a vector field along
$i_0$.   (Note that the $i_s$ may all be chosen to have the same
target; the Lie algebroid fibres are thus simpler to analyze than the
tangent spaces to $\calu\Sigma$ to which the anchor  projects them.)
 Using the lorentzian metric on the target manifold to 
decompose this vector field into its tangential and normal
components, and dividing the latter by the unit future normal field
$\mathbf n$, we obtain a vector field $X$ and smooth function $\phi$ on
$\Sigma$ which correspond to the Lie algebroid element given by the
path.   

Each fibre of $A\cale\Sigma$ is now identified with the fixed space
$\calX\Sigma \oplus \calF\Sigma$.   We omit here the verification that
this identification depends smoothly on the base point in $\calu\Sigma$.   
\end{proof}

\subsection{Gaussian normal form}
To compute the bracket and anchor of the Lie algebroid
$A\cale\Sigma\rightarrow\calu\Sigma$, we may work in a
neighborhood of the units $[i,i]$.  There, we may 
use the simplified representation of each $\Sigma$-universe $[i]$ near
$S=i(\Sigma)$ given by the following
gaussian normal form.
Using the metric $g$ on the target  of the cooriented
embedding $ M \stackrel{i}{\hookleftarrow} \Sigma$, we first 
extend the  unit future normal
field along
$S$ by parallel translation along the geodesics normal to $S$
 to obtain a time-like vector field $\mathbf{n}$ on a neighborhood $U$
of $S$ in $M$; this extension will satisfy the equations
$g(\mathbf{n}, TS) = 0$, $g(\mathbf{n},
 \mathbf{n}) = -1$, and $\nabla_\mathbf{n} \mathbf{n} = 0$, where
 $\nabla$ is the Levi-Civita connection of $g$. Transporting $S$
 along the flow $\Phi^{\mathbf{n}}_s$ of $\mathbf{n}$ produces a
 codimension-1 foliation on a neighborhood of $S$
with space-like leaves which are everywhere
 orthogonal to $\mathbf{n}$. We call the leaves of this 
 foliation the {\bf gaussian time slices}. 
This construction induces a canonical isometry to a neighborhood of
$S$ in $M$ from a neighborhood of $\Sigma \times \{0\}$ in $\Sigma \times \bbR$ on
which the metric has the
gaussian
 form
\begin{equation}
\label{eq-gaussian}
  \frac{1}{2}( \gamma_{ij} (x,t) dx^i dx^j - dt^2) \,,
\end{equation}
where $x^i$ are coordinates on $\Sigma$. 
Replacing the former neighborhood by the latter, we have established the following
normal form result.

\begin{prop}
  Every $\Sigma$-universe has a representative 
$M  \stackrel{i}{\hookleftarrow} \Sigma$ in which 
a neighborhood $U$ of $i(\Sigma)$ in $M$ is equal to a neighborhood of 
$\Sigma\times \{0\}$ in  $\Sigma\times \reals$, $i(x)=(x,0)$, and the
metric on $U$ has the gaussian form \eqref{eq-gaussian}.  This gaussian metric is uniquely
determined by $[i]$.  
\end{prop} 

Similarly, any tangent vector to $\calu\Sigma$ has a
unique representation in the form 
\begin{equation}
  \frac{1}{2} \alpha_{ij} (x,t) dx^i dx^j  \,.
\end{equation}

When $\Sigma$ is compact (and sometimes when it is not), one can
take $U$ to be a product $\Sigma \times I$ for some open
interval $I$ containing zero.  We will
call such a $\Sigma$-universe \emph{cylindrical}.  
The ``spatial'' component  $\frac{1}{2} \gamma_{ij} (x,t) dx^i dx^j$
of the gaussian metric is then a 1-parameter
family of riemannian metrics on $\Sigma$.
Thus
we may parametrize the cylindrical $\Sigma$-universes by paths of metrics, and
the tangent vectors to them by 
paths of symmetric covariant 2-tensors (not necessarily positive
definite) on $\Sigma$.    

For noncompact $\Sigma$, we may have 
to take $U$ to consist of pairs $(x,t)$ with 
$|t| < \epsilon(x)$ for a smooth positive function $\epsilon$ on $\Sigma$.
In any case, on a neighborhood of any compact subset of
 $\Sigma$, a
$\Sigma$-universe is defined by a
path of metrics and a tangent vector to the $\Sigma$-universes by a path of
symmetric covariant 2-tensors.

\subsection{Gaussian vector fields}
To use the gaussian normal form in our computations, we must 
deal with the fact that a
slicing which is gaussian for one embedding of $\Sigma$ is not
gaussian for most others.   The following lemma about diffeomorphisms will lead us to its
infinitesimal version about vector fields, which is all we need to use.

\begin{lemma}
\label{obs:gaussextend1}
Every diffeomorphism $S \rightarrow S'$ between space-like,
cooriented
hypersurfaces in spacetimes $(M,g)$ and $(M',g')$ 
extends to a diffeomorphism
$\psi: U \rightarrow U'$ between neighborhoods 
of $S$ and $S'$ respectively which respects the gaussian time-splittings,
i.e. which intertwines the (local) gaussian time flows: 
\begin{equation}
\label{eq:diffgauss}
  \psi \circ \Phi_t^{\textbf{n}} = \Phi_t^{\textbf{n}'} \circ \psi \,.
\end{equation}
The diffeomorphism is unique up to the choice of (connected relative
to $S$) $U$.
\end{lemma}

We note that (\ref{eq:diffgauss}) holds if and only if $\psi$
preserves inner products with the unit normal, i.e. $g(\mathbf{n},w) =
(\psi^*g')(\mathbf{n},w)$ for all vector fields $w$.  By letting
$(M,g)=(M',g')$ and defining
$\psi$ in the domain of a flow $\Phi^v_s$ 
generated by some vector field $v$ on $M$, and
differentiating with respect to $s$, we obtain the following
infinitesimal version of Eq.~(\ref{eq:diffgauss}). 

\begin{dfn}
  Let $U$ be a neighborhood of a hypersurface
$S$ as in Lemma \ref{obs:gaussextend1}.
A vector field $v$ on $U$ is called
  $g$-\emph{gaussian} if it satisfies 
\begin{equation}
\label{eq:vecgauss}
  (\Lie_v g)(\mathbf{n}, w) = 0
\end{equation}
for all vector fields $w$.
\end{dfn}

The following 
infinitesimal version of Lemma \ref{obs:gaussextend1} will be proven by
a purely infinitesimal computation.

\begin{prop}
\label{obs:gaussextend2}
Every vector field $v_0$ with values in $TM$ defined on a hypersurface
$S$ as in Lemma \ref{obs:gaussextend1}
 extends to a $g$-gaussian vector
field $v$ defined on a neighborhood of $S$.
\end{prop}

\begin{proof}
 Condition~(\ref{eq:vecgauss}) can be rewritten as
\begin{equation*}
\begin{split}
  0 
  &= i_\mathbf{n} \Lie_v g 
  = (\Lie_v i_\mathbf{n} + i_{[\mathbf{n},v]} ) g
  = (d i_v + i_v d)i_\mathbf{n} g + i_{[\mathbf{n},v]} g \\
  &= d (i_v i_\mathbf{n} g) + i_{[\mathbf{n},v]} g \,, 
\end{split}
\end{equation*}
where we have used the fact that $i_{\mathbf{n}}g = - dt$ so that
$di_{\mathbf{n}}g =0$. Equivalently, 
\begin{equation}
\label{eq:GaussVec}
  [\mathbf{n},v] = - \gradg(g(\mathbf{n},v)) \,,
\end{equation}
where $\gradg$ is the gradient with respect to $g$.

Splitting the vector field $v = X + \phi \mathbf{n}$ into components
orthogonal and parallel to $\mathbf{n}$ and observing that
$[\mathbf{n},v] = \nabla_\mathbf{n} v - \nabla_v \mathbf{n} =
\nabla_\mathbf{n} v - \nabla_X \mathbf{n}$ because $\nabla_\mathbf{n}
\mathbf{n} = 0$, and that $g(\mathbf{n},v) = -\phi$ because
$g(\mathbf{n},\mathbf{n}) = -1$, we obtain
the equivalent condition  
\begin{equation*} 
\nabla_\mathbf{n} v = \nabla_X \mathbf{n} + \gradg \phi \,.
\end{equation*}
If we now take the inner product with $\mathbf{n}$, use the
facts that $g(\nabla_\mathbf{n} X, \mathbf{n}) = 0$ and
$g(\nabla_\mathbf{n} \phi \mathbf{n}, \mathbf{n}) = - \mathbf{n}\cdot
\phi$, then we obtain the condition $\mathbf{n}\cdot \phi = 0$. In in
other words, the gradient $\gradg\phi$ does not have a normal
component, i.e. 
$ \gradg\phi = \gradgam \phi$, where the ``spatial gradient''
 $\gradgam \phi$ is defined as $\gradg \phi + (\mathbf{n} \cdot
 \phi) \mathbf{n}.$  (Note the sign again.)
 
This implies that that Eq.~(\ref{eq:GaussVec}) splits into components
orthogonal and parallel to $\mathbf{n}$ as 
\begin{equation}
\label{eq:bracketsplit}
  [\mathbf{n} , X ] = \gradgam\phi
  \,,\qquad
  \mathbf{n} \cdot \phi = 0 \,.
\end{equation}
In local coordinates these equations 
read 
\begin{align}\label{gaussianextensionX}
\frac{\partial X}{\partial t} &= \gradgam\phi \\ 
\label{gaussianextensionphi}
\frac{\partial\phi}{\partial t} &= 0.
\end{align} 
 Using the boundary
conditions $X_{t=0} = X_0$ and $\phi_{t=0} = \phi_0$, we see that 
the $g$-gaussian extension $v =
X + \phi \mathbf{n}$ of the vector field $v_0 =
X_0 + \phi_0 \mathbf{n}$ exists and is uniquely determined in a very
simple way by the initial values of $X$ and $\phi$. 
\end{proof}
 
We will denote the $g$-gaussian extension of $X_0 + \phi_0 \mathbf{n}$ by
$G_g(X_0,\phi_0)$.  Furthermore, we will abuse notation by omitting
the zero subscripts for the initial values when the context
distinguishes them from their extensions.

For future reference, we write below an explicit formula for the
gaussian extension, to first order in $t$.

\begin{equation}
\label{eq:gaussianfirstorder}
G_g(X_0,\phi_0) = X_0 + t \grad_{\gamma_0}\phi_0 + \phi_0 \mathbf{n} + O(t^2).
\end{equation}

\subsection{The action of $g$-gaussian vector fields on symmetric 2-forms}

To compute the anchor $A\cale\Sigma\rightarrow T\calu\Sigma$  
of our Lie algebroid, we need 
to express, in terms of the space/time splitting on the ambient manifold, 
the Lie derivative of a symmetric 2-form of the type 
\begin{equation*}
  \alpha = \frac{1}{2} \alpha_{ij}(x,t) dx^i dx^j 
\end{equation*}
by a $g$-gaussian vector field 
$v = X + \phi\mathbf{n}.$

The
pull-back of the Lie derivative on the ambient manifold to the
gaussian time slices  is 
\begin{equation*}
  \Lie^\top_v \alpha 
  := \Lie_v \alpha - (i_{\mathbf{n}}  \Lie_v \alpha) dt
  = \Lie_v \alpha - (i_{[\mathbf{n},v]} \alpha) dt \,.
\end{equation*}
If $v = X + \phi \mathbf{n}$ is $g$-gaussian we have $[\mathbf{n},v] =
\gradgam
 \phi$. Writing $\Lie_{\mathbf{n}} \alpha = \Lie^\top_{\mathbf{n}} \alpha = \frac{1}{2} \phi\, \dot{\alpha}_{ij} dx^i dx^j =: \dot{\alpha}$, we obtain
\begin{equation*}
  \Lie_v \alpha = \Lie^\top_X \alpha + \phi\,\dot{\alpha} 
   + (i_{\gradgam \phi} \alpha) dt \,,
\end{equation*}
where $\Lie^\top_X \alpha = \Lie^\top_{X(t)} \alpha(t)$ is the Lie derivative on the time slice at $t$. We will drop the superscript of $\Lie^\top_X$ if it is clear from the context what $\Lie_X$ denotes.

Using the last equation,
 we can compute the action of $v$ on the metric,
\begin{equation*}
\begin{split}
  \Lie_v g 
  &= \Lie_v (\gamma - \tfrac{1}{2}dt^2 ) 
  = \Lie_X \gamma + \phi\,\dot{\gamma} 
   + (i_{\gradgam \phi} \gamma) dt
   - d\phi \, dt \\
  &= \Lie_X \gamma + \phi\,\dot{\gamma} \,.
\end{split}
\end{equation*}
Note that the terms containing derivatives of $\phi$
cancel. Furthermore, we have 
\begin{equation*}
  \Lie_{\mathbf{n}} g = -2K \,,
\end{equation*}
where $K(X,Y) = g(\nabla_X Y, \mathbf{n})$ is the second fundamental
form with respect to $\mathbf{n}$ of the time slice at $t$. 
(See Section 9.3 of \cite{Wald}, but note that Wald's ``extrinsic
curvature'' is the negative of the second fundamental form.)   Thus,
the Lie derivative of the metric with respect to a
$g$-gaussian vector field is given by 
\begin{equation}
\label{eq:gaussact}
  \Lie_{X + \phi \mathbf{n}} g 
  = \Lie_X \gamma + \phi \dot{\gamma} = \Lie_X \gamma - 2\phi K \,.
\end{equation} 

\subsection{The bracket and the anchor}
\label{subsec-bracketanchor}

We can now make explicit the Lie algebroid structure on 
$A\cale\Sigma \to \calu\Sigma$
in terms of the trivialization $A\cale\Sigma \approx 
\calu\Sigma \times (\calX\Sigma \oplus \calF\Sigma)$ 
given by gaussian extension. 

Using (\ref{eq:bracketsplit}), we find the 
Lie bracket of two $g$-gaussian vector
fields to be:
\begin{equation}
\label{eq:Katzbrack}
  [X + \phi\mathbf{n}, Y + \psi\mathbf{n}]
  = [X, Y] + \phi \,\gradgam\psi - \psi \,\gradgam \phi 
   + (X \cdot \psi - Y \cdot \phi) \mathbf{n} \,,
\end{equation}
which corresponds exactly to 
(\ref{brackets}). 

As for any bracket of vector fields, the jacobiator 
$[u,[v,w]] + [v,[w,u]] +[w,[u,v]]$
  of three $g$-gaussian vector fields $u$, $v$, and $w$
  vanishes. However, the bracket of two
  $g$-gaussian vector fields is in general not
  $g$-gaussian, as the following proposition shows.

\begin{prop}
  The bracket (\ref{eq:Katzbrack}) of two $g$-gaussian vector fields is 
  not always $g$-gaussian. 
\end{prop}

\begin{proof}
We have
\begin{equation*}
\begin{split}  
  i_\mathbf{n} \Lie_{[v,w]}g
  &= \Lie_{[v,w]} i_\mathbf{n} g + i_{[\mathbf{n},[v,w]]} g \\
  &= d(g([v,w],\mathbf{n})) 
      + i_{[v, [\mathbf{n},w]]} g + i_{[[\mathbf{n},v],w]}g \\
  &= - d( X \cdot \psi - Y \cdot \phi)
      + i_{[v,\gradgam \psi]}g - i_{[w,\gradgam \phi]} g \\
  &= - d( X \cdot \psi - Y \cdot \phi)
      + (\Lie_v i_{\gradgam \psi} - i_{\gradgam \psi} \Lie_v) g
      - (\Lie_w i_{\gradgam \phi} - i_{\gradgam \phi} \Lie_w) g \\
  &= - d( X \cdot \psi - Y \cdot \phi) + \Lie_v d\psi - \Lie_w d\phi
      + i_{\gradgam \phi} \Lie_w g - i_{\gradgam \psi} \Lie_v g \\
  &= i_{\gradgam\phi}(\Lie_Y \gamma - 2 \psi K) 
    - i_{\gradgam \psi}(\Lie_X \gamma - 2 \phi K) \\
  &= i_{\gradgam \phi} \Lie_Y \gamma - i_{\gradgam\psi} \Lie_X \gamma
      + 2 i_{(\phi \,\gradgam \psi - \psi\, \gradgam \phi)} K 
\,.
\end{split} 
\end{equation*}
In the second step, we have used that $i_\mathbf{n}g = -dt$. On the
hypersurface $\Sigma \times \{0\}$, we may choose $X$, $\phi$, $Y$, and
$\psi$ arbitrarily.  For $X = Y = 0$ there, 
the right hand side of the last equation is the second fundamental
form contracted with $\psi\, \gradgam \phi - \phi\, \gradgam\psi$,
which is generally not zero. 
\end{proof}

The anchor $A\cale\Sigma\stackrel{\rho}{\rightarrow} T\calu\Sigma$
of our Lie algebroid is given, up to a sign by the action computed in the previous
section:
\begin{equation}
\label{eq:action1}
  \rho(X, \phi, g) := -\Lie_{X(t) + \phi \mathbf{n}} g 
  = -\Lie_{G_g (X,\phi)} \gamma - \phi \, \dot{\gamma} \,,
\end{equation}
where $g = \gamma(t) - \frac{1}{2} dt^2$ is a gaussian metric on
$\Sigma \times \reals$ and $G_g (X,\phi)=X(t) + \phi \mathbf{n}$ is the $g$-gaussian
extension of $X + \phi \mathbf{n}$.  This Lie derivative represents
the change in the ``appearance'' of the metric on the ambient manifold
as the ``viewpoint'' changes 
according to a vector field $X+\phi \mathbf{n}$ along a space-like embedding. 
Note that the anchors of two constant sections are applied consecutively to a metric as $\rho(X,\phi) \rho(Y,\psi) (g) = \calL_{G_g (Y,\psi)} \calL_{G_g(X,\phi)} g$, that is, by applying the Lie derivatives of the gaussian extensions in reverse order. This takes care of the negative signs in Eq. \eqref{eq:action1}.

The kernel of $\rho$ consists of
those $(X,\phi)$ whose $g$-gaussian extension is a Killing
vector field. It follows that the dimension of this kernel is 
zero\footnote{The dimension of the kernel of $\rho$ is always
finite; it is at most $\frac{1}{2}(n+1)n + n+1$, where $n=\dim
\Sigma$, with equality only when $g$ has constant sectional
curvature.} over an open dense subset of $\calu\Sigma$,
since any $\Sigma$-universe is contained in a 1-parameter family whose
generic members have no isometries.  (Local
perturbations of the lorentzian metric
 suffice to achieve this; deeper results in the riemannian case go
back at least as far as \cite{eb:manifold}.)
It follows immediately that
the bracket on sections of the
Lie algebroid $A\cale\Sigma$ 
is determined by that on their images under the
anchor (\ref{eq:action1}). 
In particular, the bracket of constant sections
$(X,\phi)$ and $(Y,\psi)$ of the trivial bundle must be that given by
given by Eq.~(\ref{eq:Katzbrack}):
\begin{equation}
\label{eq:Katzbracket1}
  [(X,\phi), (Y, \psi)]
  = ([X, Y] + \phi \,\gradgam\psi - \psi \,\gradgam \phi, 
   X \cdot \psi - Y \cdot \phi) \,,
\end{equation}
where we now evaluate $X$ and $\phi$ at $t=0$.  
Together with the anchor
(\ref{eq:action1}), this bracket determines the Lie
algebroid structure on $A\calu\Sigma$.  
To summarize, we have:

\begin{thm}
\label{thm-algebroidoveruniverses}
  The 
  $\Sigma$-evolutions $\cale \Sigma$ form a groupoid over the 
$\Sigma$-universes  $\calu\Sigma$.
Each orbit of this groupoid consists of all $\Sigma$-universes
which are represented by $\Sigma$-spaces
in a fixed spacetime
$(M,g)$.
  The Lie
  algebroid $A\cale\Sigma$ has a natural identification with the 
 trivial bundle $ \mathcal{U}\Sigma \times (\calX\Sigma \oplus
  \calF\Sigma).$  Under this identification,
in the gaussian representation of ambient metrics,  
the anchor is the Lie derivative by gaussian extensions:
  $
\rho(X,
  \phi, g) = \calL_{G_g (X,\phi)}g.$
 The bracket of constant
  sections is given by (\ref{eq:Katzbracket1}).
\end{thm}

\subsection{An equivalent groupoid and the moduli stack of spacetimes}
\label{subsec-stack}
This section is peripheral to the main argument of this paper, but it
suggests another point of view toward the groupoid $\cale\Sigma$ of
$\Sigma$-evolutions.  

We have just seen that the orbits of $\cale\Sigma$
are in one-to-one correspondence
with isometry classes of $\Sigma$-adapted spacetimes, while the
isotropy groups in $\cale\Sigma$ are just the isometry groups of those
spacetimes.  It turns out that $\cale\Sigma$ is equivalent to another
groupoid in which $\Sigma$ plays a much less central role.   We
only require that the spacetimes be $\Sigma$-adapted, without ever
specifying the placement of $\Sigma$.   For the notion of equivalence
of groupoids, we refer to \cite{bl-we:grouplike} and  \cite{mo-mr:lie}.

\begin{dfn}
The groupoid $\cali\Sigma$ is defined to be that in which the objects
are $\Sigma$-adapted spacetimes, and the morphisms are isometries
between these spacetimes.  
\end{dfn}

\begin{rmk}
Strictly speaking, $\cali\Sigma$ is not a groupoid, since the spacetimes do not
form a set.  But it is equivalent to its wide subgroupoid consisting
of isometries between
those spacetimes whose underlying manifolds are
submanifolds of euclidean spaces.
\end{rmk}

$\cali\Sigma$ may be identified with an action groupoid whose objects are
the elements of the collection
$\cals\Sigma$ of all $\Sigma$-adapted spacetimes
(or, alternatively, a set as in the remark above).  The groupoid acting on
$\cals\Sigma$ consists of the diffeomorphisms between these objects.
If $(M_1,g_1)$ is a $\Sigma$-adapted spacetime and $\Phi$ is a diffeomorphism
to a manifold $M_0$ from $M_1$, then the result of acting on
$(M_1,g_1)$ by $\Phi$ is defined to be the spacetime $(M_0,(\Phi^{-1})^*
g_1)$.  $\Phi$ then becomes an isometry and so may be considered as a
morphism in $\cali\Sigma$.   

\begin{prop}
The groupoids $\cali\Sigma$ and $\cale\Sigma$ are equivalent.
\end{prop}

\begin{proof}
An equivalence between
groupoids is given by a biprincipal bundle, i.e. a space on which
the groupoids have commuting free actions, with the fibres of the moment
map\footnote{An action of a groupoid $G$ over $G_0$ on a space $X$
  includes as part of its data a map $G_0 \from X$ which determines
  which groupoid elements act on which elements of $X$.  This map is
  sometimes called the {\bf moment map} of the action.}
 of each one being the orbits of the other.   Such a bibundle
induces bijections between the orbit spaces of the two groupoids and
between the isotropy groups of corresponding objects.   

To get an equivalence between $\cali\Sigma$ and
$\cale\Sigma$, we take as total space of our bundle the collection
$\calh\Sigma$ of all $\Sigma$-spaces in all possible spacetimes.

The moment map for the left action of $\cali\Sigma$
forgets the embedding and remembers only the target.  
The typical morphism in $\cali\Sigma$ is an isometry 
$M' \stackrel {\psi}\from M$ 
between spacetimes.  It acts on any $\Sigma$-space
$M \stackrel{i}{\hookleftarrow}\Sigma$ to give the (equivalent) embedding 
$M'\stackrel{\psi\circ i}{\hookleftarrow} M$.  This action is free because an
isometry of a connected manifold 
which fixes a hypersurface and its normal bundle must be the
identity.

The moment map for the right action of $\cale\Sigma$
takes each  $M'\stackrel{\psi}{\hookleftarrow} M$
to its equivalence class $[i]$. 
If $i \in \calh\call$ and $[i_1,i_0]$ are such that $[i]=[i_1]$, then
we define the right action of $[i_1,i_0]$ on $i$  as
follows.  The equivalence between $M\stackrel{i}{\hookleftarrow}\Sigma$ and 
$M'\stackrel{i'}{\hookleftarrow}\Sigma$
is realized by a unique isometry $M'\stackrel{\psi}{\from}M$
such that $\psi\circ i_1 = i$.  Then $i \cdot [i_1,i_0]$ is defined
to be the embedding $M'\stackrel{\psi\circ i_0}{\hookleftarrow} \Sigma$.  
To see that this action is free, suppose that $i \cdot [i_1,i_0] =
i$.  Then $M=M'$ and $\psi\circ i_0 = i = \psi \circ i_1,$ so $i_1 =
i_0$, and $[i_1,i_0]$ is an identity morphism.

The transitivity of the left action on the right moment fibres is just
a restatement of 
the definition of equivalence used in defining the
$\Sigma$-universes.  Transitivity of the right action on the left
moment fibres is obvious, since the morphisms in $\cale\Sigma$ are
(equivalence classes) of pairs of embeddings into the same
target.
\end{proof}

\begin{rmk}
The notion of {\bf stack} was introduced in algebraic geometry 
and has recently migrated to differential geometry
\cite{be-xu:differentiable}; the purpose of the notion is to provide a
description of spaces of equivalence classes when it is important to
keep track of the multiple ways in which objects can be equivalent and
thereby to overcome difficulties related to singular behavior of
quotient spaces.  
One way to understand stacks is to see them as {\bf presented} by
groupoids, where equivalent groupoids determine ``the same'' stack in
the same sense that a given manifold may be described in different
ways by overlapping families of coordinate charts.

The equivalence between $\cale\Sigma$ and $\cali\Sigma$ which we have
just proven shows that we
may consider $\cale\Sigma$ as a presentation of the moduli stack of
$\Sigma$-adapted spacetimes, i.e. the isometry classes (including
information about self-isometries) of spacetimes admitting 
$\Sigma$ as an ``instantaneous space'', or
``initial condition''.  (Note
that this is a purely ``kinematic'' construction, as 
we have not imposed any dynamical condition such as the
vacuum Einstein
equations.)
\end{rmk}

\section{Discussion: the descent problem}

We have constructed a groupoid over $\calu\Sigma$ whose Lie algebroid
bracket, for constant sections in a natural local trivialization,
exactly matches the bracket relations on the constraints for
Einstein's equations.  To establish a more direct relation with the
constraints themselves, we would need to find similar structures on
the phase space $T^*\calm\Sigma$.

There is a natural projection $\calp$ to $ T\calm\Sigma$ from $\calu\Sigma$,
assigning to every $\Sigma$-universe the 1-jet with respect to $t$
at $t=0$ of its gaussian representation as a path in $\calm\Sigma$.
(Note that this is well defined even if the $\Sigma$-universe is not
cylindrical.)
But there is no way to push our Lie algebroid
forward under this projection, essentially because the value of the
anchor at a given 1-jet would have to depend on the 2-jet.  

To
surmount this difficulty, we tried to use a second-order evolution equation
on $\calm\Sigma$, i.e. a rule which expresses 2-jets in terms of
1-jets, such as the Einstein evolution equations themselves.  But the
resulting anchor was not consistent with the bracket relations.

We also tried ``reverse engineering'', \emph{defining} the anchor so
that it would take constant sections to the hamiltonian vector fields
of the constraint functions.   But the anchor so-defined, when applied to
two constant sections associated to functions $\phi$ and $\psi$ on $M$, 
generally takes their bracket as defined by
the metric-dependent relations \eqref{brackets}
to a vector field which is not even hamiltonian.  

\section{Some history}
\label{subsec-history}
A hamiltonian formulation of general relativity can be found in the
work of Pirani and Schild \cite{Pi-Sch}, aimed at the quantization of
Einstein's gravitational field equations.  These authors, as well as
Bergmann, Penfield, Schiller, and Zatzkis \cite{Berg-Pen-Schil_Zat}
consider a physical state at a certain time to be given by data on a
space-like hypersurface, which must in some sense be arbitrary in
order to maintain four-dimensional covariance.

Bergmann \cite{Berg} had by then already begun a systematic study of
covariant field theories of general type, addressing the problem of
bringing general relativity into the canonical form as a preliminary
step to quantization. While canonical quantization of field theories
was being developed, it soon became clear that general relativity
posed additional difficulties connected with the degeneracy of the
lagrangian which was a consequence of four-dimensional diffeomorphism
invariance.  In fact, Dirac's original work \cite{Dirlect} on
constrained dynamics was inspired in large part by this problem.

It was noted by Pirani, Schild, and Skinner \cite{Pi-Sch-Ski} and not
long after by Dirac \cite{di:GravHam} and Arnowitt, Deser, and Misner
(ADM) \cite{ADMS} that great simplifications could be made at the
expense of giving up four-dimensional symmetry.  Such simplifications
became possible by fixing a foliation of spacetime by space-like
hypersurfaces, with the physical states living on these surfaces.
Such a spacetime decomposition leads to the decomposition of vectors
along each hypersurface into their normal and tangential components,
and the metric tensor itself may be presented in the form
\[
g^{(4)}_{\mu\nu}=
\left(\begin{array}{cc}
  N^{2}+N_{s}N^{s} & N_{n} \\
  N_{m} & g_{mn}
\end{array}\right) \,,
\]
where the lapse function $N$ and the shift 3-vector $N^n$ describe the
variation of the time and space coordinates on infinitesimally close
space-like hypersurfaces.   The lapse and shift can be chosen
arbitrarily but enter in the constraint equations, which arise from
the degeneracy of the lagrangian and are given by time components of
the Einstein field equations $G^{0\mu}=0$. It had already been
realized by Dirac and many others that the shift functions $N_{n}$
generated coordinate transformations on the hypersurfaces, whereas the
lapse $N$ was related to time translation. In this sense, lapse and
shift could be viewed as gauge potentials that had to be fixed in
order to solve the initial value problem. The lapse and shift
functions were first introduced in \cite{ADM2}. Their geometrical
meaning was explained in \cite{Wheeler} and in more global terms by
Fischer and Marsden \cite{FisherMarsden}, who were 
perhaps the first authors
 to suggest that the constraints should be seen as something like a momentum map.

 Still in search of avoiding the shortcomings of coordinate-dependent
 language, Kucha\v{r} \cite{Kuchar} argued that field dynamics does
 not take place in spacetime, or along a single foliation of
 hypersurfaces but in what he called hyperspace, an
 infinite-dimensional manifold consisting of all the spacelike
 hypersurfaces in a given spacetime.

After the appearance of \cite{di:GravHam}, Katz \cite{ka:crochets}
found the formulas for the Poisson brackets of the constraint
functions.    In their space-integrated form, the brackets were first
computed by DeWitt \cite{dw:QuantGrav}. While the computation of the
brackets was straightforward, their geometric interpretation was not
satisfactory. A number of authors have tried to give a more conceptual
derivation, e.g., by studying hypersurface deformations
\cite{Teitelboim}, as a method to guarantee the path-independence of
geometrodynamical evolution \cite{KK}, or by generalizing the concept
of transformation groups \cite{be-ko:coordinate} \cite{Wheeler}. 
In particular, Teitelboim \cite{Teitelboim} was perhaps the first 
to show that the Poisson bracket relations are purely a consequence of
the geometry of hypersurfaces in a riemannian or lorentzian manifold,
independent of any particular field theory.  In a sense, our paper may be
seen as setting Teitelboim's argument in its proper mathematical setting, that
of groupoids and Lie algebroids.

The
most ambitious approaches aimed at recovering diffeomorphism
covariance of the initial value problem
\cite{ho-ku-te:geometrodynamics} \cite{is:canonical}
\cite{IshamKuchar}
\cite{Lee-Wald}.
Parallel developments were also made by Wheeler
\cite{Wheeler}. Inspired by the dynamical description of the
electromagnetic field in terms of the vector potential, he proposed to
describe the dynamics of three-space metrics through the propagation
of the intrinsic metric of space-like hypersurfaces with respect to a
time coordinate. Such a time parameter would label the leaves the
dynamically produced spacetime. In this approach the extrinsic
curvature of these hypersurfaces corresponds to the canonical
momenta. The configuration space of this theory is known as Wheeler's
superspace. Trajectories of metrics in this space produce
four-dimensional spacetime geometries.  

Since the bracket of constraint functions is metric dependent, this
suggested to Bergmann and Komar \cite{be-ko:coordinate} that the
associated symmetries could be metric-dependent as well.  Their paper
contains many ideas which are very close to ours, although the language is
somewhat different.  In particular, we would say that their
``Q-type transformations'' are precisely the bisections of the action
groupoid (essentially our $\mathcal{I}\Sigma$, but without a specific
choice of $\Sigma$), associated to the diffeomorphisms of a 4-manifold
$M$ acting on the function space of lorentzian metrics on $M$.  Their
infinitesimal transformations are the sections of the action Lie
algebroid; their equation (3.1) is precisely the formula for the
bracket of sections in this Lie algebroid!  Bergmann and Komar also
observe that the orbits are isometry classes and that the action fails
to be faithful in the presence of isometries.

Bergmann and Komar even make the tantalizing statement, ``That these
transformations form a group, or at least a groupoid, is seen from
their definition.'' Unfortunately, groupoids do not reappear anywhere
in the paper, and it is not clear what notion of groupoid the authors
had in mind.  In particular, although \cite{be-ko:coordinate} includes
a discussion of the set of diffeomorphisms between hypersurfaces in
spacetime (the infinitesimal transformations being 4-vector fields
along these hypersurfaces), there is no suggestion that these form a
groupoid.

The idea of associating the Q-type transformations with diffeomorphism
invariance of general relativity appeared also in \cite{Castell} where
these transformations arise from ``field dependent'' gauge
generators. These field dependent generators appeared also in
 \cite{Pons}, \cite{PonSa2},  \cite{PonSa1}, and \cite{Salis},
and more recently in \cite{Muk}.

Hojman, Kucha\v{r}, and Teitelboim \cite{ho-ku-te:geometrodynamics}
look at space-like embeddings into a fixed lorentzian manifold and
find what is more or less the Lie algebroid bracket for the Lie
algebroid of the groupoid of diffeomorphisms between hypersurfaces.

Finally, we should at least mention the immense analysis literature on
existence theory, both local and global, for the Einstein initial
value problem, beginning with fundamental work of Lichnerowicz
\cite{li:problemes} and Four\`es-Bruhat ($=$ Choquet-Bruhat)
\cite{fo:theoreme} and continuing to this day.  (See \cite{ch-results}
for a fairly recent survey.)  We hope that our work provides new
geometric understanding which may contribute to both the analysis and
the quantization of the Einstein field equations.

\appendix
\section{Diffeology}
Although some of the infinite-dimensional
spaces of smooth mappings in this paper
may be considered as Fr\'echet manifolds, this analytical structure
is not necessary for formal computations.
Instead, we work in the framework of diffeological spaces.  
These objects were introduced\footnote{A very similar notion was
introduced by Chen \cite{ch:iterated}, and different notions of
 ``smootheology'' are
compared in \cite{st:comparative}} by Souriau \cite{so:groupes}
and developed extensively by Iglesias \cite{ig:these} and others.
We give a brief introduction to diffeology here and refer
to \cite{he-ma:diffeological}, \cite{le:diffeological}, and
the work-in-progress \cite{ig:diffeology} for further details.  
Similar notions in the topological setting are discussed 
in \cite{ay:geometric}.   

Roughly speaking, we can do a great deal of differential geometry 
on a set $X$ once we know what it means for a family of elements of $X$ to
depend smoothly on parameters.    

\begin{dfn}
A {\bf parameter space} is an open subset of $\reals^n$ for some $n$,
and a {\bf parametrization} of
a set $X$ is a  map to $X$ from some parameter space $P$.
A {\bf diffeology} on $X$ is a set $\cald$ of parametrizations of $X$ 
which contains all constant maps, which is closed
under composition on the right with smooth (in the usual sense) maps between
parameter spaces, and which is locally defined in the sense that a
parametrization is in  $\cald$ if and only if its restrictions to
all the sets in some open covering are.  The elements of $\cald$ are
called {\bf plots}, or {\bf smooth parametrizations}.
$(X,\cald)$ 
is called a {\bf diffeological space}; we denote it simply by $X$
when it is clear what diffeology is being used.  
A plot $X\stackrel{\phi}{\from} P$ with $x=\phi(0)$  is 
called a {\bf plot at} $x$.
\end{dfn}

{
\begin{rmk}
  In the language of sheaf theory, a diffeological space is a concrete
  sheaf on the site of open subsets of euclidean spaces
  \cite{BaezHoffnung:Convenient}. This point of view is particularly
  well suited to show that the category of diffeological spaces has
  small limits (taken point-wise), small colimits (first taken
  point-wise, then sheafified), and exponential objects (given by the
  universal property), thus allowing for constructions such as
  quotient spaces, pull-backs, mapping spaces, etc. that generally
  fail to exist for smooth manifolds. For our purposes, however, we
  will need to give the explicit descriptions of these constructions.
\end{rmk}
}

Just as in topology, every set $X$ carries the {\bf discrete} diffeology,
for which only the locally constant maps are plots, and the {\bf coarse}
diffeology, for which every parametrization is smooth.
If $X$ is a (finite-dimensional) manifold, the usual smooth maps
to $X$ from parameter spaces form the ``standard'' diffeology.

Diffeological spaces are the objects of a category in which the
morphisms  are the smooth maps, defined as follows.

\begin{dfn}
A map $X\stackrel{f}{\from}Y$ between diffeological spaces
is a {\bf smooth map} if
$f\circ \phi$ is a plot for $X$ whenever $\phi$ is a plot for $Y$. 
We denote the set of all such smooth maps by $\cinf(X,Y)$.  (Note that
many authors denote it by $\cinf(Y,X)$.)  A smooth map with a smooth
inverse is a {\bf diffeomorphism}.
\end{dfn}

A smooth map between
manifolds with the standard diffeologies is just a smooth map in
the usual sense.   The parametrizations of any diffeological space which
are smooth are just the plots, so the term ``smooth parametrization''
has an unambiguous meaning.

Any diffeological space $(X, \cald)$ carries the $\cald$-topology, 
defined as the
finest topology for which all plots are continuous; i.e., a
subset $U\subseteq X$ is open [closed] if and only if $\phi^{-1}(U)$ is
open [closed] for every plot $\phi$.   Smooth maps are always
$\cald$-continuous.
 
A product $X\times Y$ of diffeological spaces
carries a  {\bf product diffeology}, whose plots are the
parametrizations whose compositions with the projections to $X$ and
$Y$ are smooth.  Each
subset of a diffeological
space has a natural {\bf subspace diffeology} in which the plots are those
parametrizations whose composition with the inclusion is smooth.
The restriction of a smooth map to any subset with the subspace
diffeology is again smooth.

Any set $\cald_0$ of maps to $X$ from diffeological spaces
{\bf generates} the diffeology $\overline{\cald}_0$ consisting of all those
parametrizations which are locally  
compositions of the form $\phi\circ s$, where $\phi$ is in  $\cald_0$ and 
$s$ is a parametrization of the domain of $\phi$, together with all
constant maps from parameter spaces.   (The latter are already
included if the images of the elements of $\cald_0$ cover $X$.)

If $\cald_0$ consists of a single surjective map $X\from Y$ from a
diffeological space $Y$, the diffeology which it generates is called
the {\bf quotient diffeology}.   Conversely, any
diffeology is the quotient diffeology for the union of all
its plots, considered as a single map 
defined on the disjoint union of the domains of the plots.  

  If $X$ and $Y$ are
diffeological spaces, the {\bf functional diffeology}
 on $\cinf(X,Y)$ is  that for which 
a parametrization $\cinf(X,Y)\stackrel{~\phi}{\from}P$ is a
plot if and only if the corresponding evaluation map $X\from Y\times P$ is
smooth.  The ``exponential law''
$\cinf(\cinf(X,Y),Z)) \cong \cinf(X,Y \times Z)$ then holds for all
$X$, $Y$, and $Z$, not just in the defining case where
$Z$ is a parameter space, and the
composition operations $\cinf(X,Z)\from \cinf(X,Y) \times \cinf(Y,Z)$
are smooth.

The subspace diffeology construction produces diffeologies
on spaces of mappings satisfying extra conditions, such as
spaces of diffeomorphisms or embeddings, spaces of sections of smooth
bundles (e.g. tensors),
and  solution spaces of ordinary or partial differential equations.

We can also define diffeologies on spaces of mappings with variable
domains.  For simplicity, we let these domains be subsets of a fixed
space.  Let $[Y]$ be some collection of subsets of a diffeological
space $Y$ (for instance the open subsets, if $Y$ carries a topology),
and let $\cinf(X,[Y])$ be the set of all mappings to $X$ whose domains
belong to $[Y]$.  Thinking of a parametrization 
$\cinf(X,[Y])\from P$ 
as a family $X \stackrel{f_p}{\from}Y_p$ of maps
parametrized by $p\in P$, we call it a plot when the evaluation map
$X \from D \subseteq Y \times P$ is smooth, where the domain 
$D=\{(y,p)|y\in Y_p)\}$ carries the subspace diffeology.  
When $Y$ is a manifold and $[Y]$ consists of the open subsets,
it may be appropriate to restrict the diffeology to consist of those
families for which the domain $D$ is open in $Y\times P$, so
that we are always dealing with maps defined on  manifolds.  In this case,
we denote the space of mappings by $\cinf(X,Y_{\mathrm{open}})$.
In particular, if $X$ is a single point, we obtain a diffeology on the
set of open subsets of $Y$.

\begin{ex}
\label{ex-openmetric}
If $\Sigma$ is any manifold, the space $\calm\Sigma$ of riemannian
metrics on $\Sigma$ carries a functional diffeology.   So does
$\calm\Sigma_{\mathrm{open}}$, the metrics defined on open subsets of $\Sigma$.
\end{ex}

If the diffeological space $Y$ carries a
 topology (not necessarily the $\cald$-topology), $\cinf(X,Y)$ also
 carries the  
``finer'' {\bf compact functional
  diffeology} in which the plots 
$\cinf(X,Y)\stackrel{~\phi}{\from}P$ are required to satisfy the additional
condition that each $p\in P$ has a neighborhood $\calv$ for which 
all the maps in $\phi(\calv)$ agree outside some compact subset of $Y$.  
$\cinf(X,Y)$ with this diffeology is denoted by $\cinf_c(X,Y)$.

\subsection{Jets and tangent vectors}

To define jets in diffeology, we start with the basic case of
parameter spaces.  For $k=0,1,2,\ldots,\infty$,
and smooth maps $s$ and $t$
to $\reals^n$ from a parameter space $P$ containing $0$, 
$s$ and $t$ are defined to
have the same $k$-jet at $0$   
if $s(0)=t(0)$ and if their partial derivatives  through order
$k$ match at $0$.  Plots 
$f$ and $g$ to a diffeological space $X$ from $P$
will be said to have the same 
$k$-jet at 0 if there is a
plot $X\stackrel{h}{\from} Q$ such that $f= h\circ s$ and $g=h \circ t$, where
$s$ and $t$ are plots for $Q$
 the same $k$-jet at $0$.
Finally, for any diffeological spaces $X$ and $Y$, two maps
$X\from Y$
have the same $k$-jet at $y\in Y$ if their compositions with
any plot at $y$ have the same $k$-jet at 0.  (It is easy to see 
that, if $X$ and $Y$ are parameter spaces, this coincides with the
original definition.)
 
Having the same $k$-jet at $y$
is an equivalence relation $\sim ^k_y$, and two maps with the same
germ at $y$ have the same $k$-jet there for any $k$.   We thus obtain an
equivalence relation on germs, and
we  define the space
$J^k(X,Y)$ of $k$-jets of maps to $X$ from $Y$ to be the set of
pairs $(j^k_y f,y)$, where $f$ is the germ of a map to $X$ from some neighborhood
of $y$ in $y$, and  $j^k_yf$ is the equivalence class of $f$
for $\sim^k_y$.   

If $K\subset Y$, we say that $f\sim^k_K g$ if $f$ and $g$ have the
same $k$-jet at all points of $K$, and we call the equivalence
classes for this relation $k$-jets along $k$.  If $X$ is a bundle over
$Y$, we may refer to the $k$-jets of those maps $Y\from X$ which
happen to be
sections of the bundle as $k$-jets of sections.

We define diffeological structures on jet bundles by considering
them as quotients of spaecs of mappings, using the functional diffeology.
When $k=\infty$ and $X$ and $Y$ are manifolds, the sheaf diffeology is
also interesting, since the smooth maps are maps into the leaves of
a foliation.

Since constant maps are always smooth, and the $0$-jets of maps are
just their values, there is a natural identification of $J^0(X,Y)$
with $X\times Y$.
There are also natural maps $J^k(X,Y)\from J^l(X,Y)$ for $k\leq
l$.  For $k = 0$, this gives natural projections of the jet spaces to
$X$ and $Y$.  That the
 infinite jet space $J^\infty(X,Y)$ is the inverse limit of the jet
 spaces for finite $k$.    

Jets of mappings into mapping spaces are mappings into jet bundles.
If $X$, $Y$, and $Z$ are manifolds,
smooth maps
$C^\infty (X,Y)\stackrel{f}{\from} Z$ correspond to smooth maps 
$X\stackrel{F}{\from} Y\times Z.$  It follows 
that 
$J^k(C^\infty(X,Y), Z) = C^\infty (J^k(X,Z),Y).$

Of special importance are the $1$-jets at $0$ of
 maps to $X$ from neighborhoods of $0$ in $\reals$; these are the
tangent vectors.  We denote the set of all tangent vectors to $X$ by
$TX$ and call it the {\bf tangent cone bundle} of $X$.  It is the
disjoint union of tangent cones $T_x X$ at the points of $X$, which
are cones because reparametrization of curves by the action of the
multiplicative group $\reals \backslash \{0\}$ on $\reals$ leads to a
natural action of this group on the tangent spaces.  $TX$ has a
diffeological structure and projection to $X$ inherited from those on
$J^1(X,\reals)$.  

For mapping spaces between manifolds, the general result above on jet
bundles gives, with $Z=\reals$ and $k=1$,
$T C^\infty(X,Y) = C^\infty(TX,Y)$.  The tangent bundle projection
$C^\infty(X,Y)\from TC^\infty(X,Y)$ is just composition with the
projection $X\from TX$.  

Two warnings are in order.  First of all, the action of $\reals
\backslash \{0\}$ on
the tangent cone may
not be faithful.  For instance, if $X$ is the
half-line $[0,\infty)$ viewed as
$\reals / \integers_2$ with the quotient diffeology, its tangent cone at
$0$ is also a half-line on which multiplication by $-1$ is the
identity.  Second, it is generally not possible to add $1$-jets of
curves.  If it were, we would find that adding a vector $v$ at $0$ in
$\reals / \integers_2$ to itself would give both $2v$ and $0$.  Another example 
where addition is not possible is the tangent cone at $0$
to the union of the coordinate axes in $\reals^2$ with the subspace diffeology.

An example where all the tangent cones are zero is the space
$[Y]_{\rm{open}}=\cinf(\{p\},[Y])_{\rm{open}}$ of open subsets of a manifold $Y$.  A
path in $[Y]_{\rm{open}}$ is just an open subset of $I \times Y$ for
some interval $I$ of real numbers.  If $U_s$ is any such path defined
around $s=0$, it has the same tangent vector as the constant path through
$U_0$.  In fact, they factor via curves tangent at $0$ 
through the plot on $I \times \reals$ defined by 
$U_{(s,0)} = U_s$, $U_{(s,s^2)} = U_0$,
and $U_{(s,s')} = Y$ elsewhere.  
Similarly, tangent vectors to spaces $\cinf(X,[Y])_{\rm{open}}$
are insensitive to the variation of domain with $s$ and may
all be represented by families of functions with unchanging domains.
The same arguments apply to jets of any order.

\begin{ex}
\label{ex-metrics}
For the space $\calm\Sigma_{\mathrm{open}}$ of riemannian metrics
defined on open subsets of $\Sigma$, the tangent space at a metric $g$
defined on all of $\Sigma$ consists of the smooth symmetric covariant
2-forms on $\Sigma$, even though the domain of a path through $g$ may
shrink as the path parameter varies.
\end{ex}

It is possible to define a tangent vector {\em space} at each point of a
diffeological space by taking formal linear combinations of $1$-jets
of curves, as in \cite{he-ma:diffeological}, but it is then necessary to
introduce further relations, so that, for instance, the tangent space
at the conical singular point $0$ of the space $\reals/\integers_2$ 
reduces to zero, while the tangent cone does not.  (On the other hand,
the tangent space at the intersection point in the union
of the coordinate axes in the plane is
two-dimensional.)  

For a diffeological group, group multiplication induces a
 a vector space
structure on each tangent cone
\cite{le:diffeological}.

Cotangent spaces to diffeological spaces may be defined as spaces of
1-jets of smooth mappings to $\reals$; these all have vector space
structures derived from that on $\reals$.  The composition of
real-valued functions with curves defines a natural pairing (which can
be degenerate) between
tangent and cotangent spaces.

If $N$ and $E$ are
manifolds and $p$ is a submersion, then the finite jet spaces $J^k(Y,E)$ 
are also smooth manifolds with submersions to $Y$, and the infinite
jet space $J^\infty(Y,E)$ carries the projective limit diffeology in
which a map into it is smooth if all of the compositions with
projections into finite jet spaces are smooth.

\subsection{Diffeological groupoids}
\label{sec-diffgroupoid}

Diffeological groupoids are defined like
 topological groupoids   
\cite{ig:these}\cite{ig:diffeology}.
Recall that a category $C$
consists of a collection $C_0$ of objects and a collection $C_1$ of morphisms, with
target and source maps $l$ and $r$ to 
$C_0$ from $C_1$, a unit inclusion map $C_1 \stackrel{\epsilon}{\from}C_0$,
and a composition map to $C_1$ from 
$C_2=\{(f,g)\in C_1 \times C_1 |r(f)=l(g)\},$ satisfying the
usual axioms for associativity and units.  A {\bf groupoid} is a 
category\footnote{Many authors require a groupoid to be a {\bf small}
  category, in the sense the morphisms and objects form sets rather
  than just collections.  We will not make this assumption, but will
  occasionally point out how to replace ``large'' groupoids by small ones.}
 in which every morphism is invertible.

If $C_0$ and $C_1$ are diffeological
spaces, we give $C_2 \subseteq C_1\times C_1$ the subspace diffeology.  If
all of the structure maps are smooth, we say that $C$ is a
diffeological category.   If $C$ is, in addition, a groupoid, and the
map $\iota$ which takes each morphism to its inverse is smooth, then
$C$ is a {\bf diffeological groupoid}.   Note that we do not require
$l$ and $r$ to be submersions; in fact, this notion is better replaced
with that of ``subduction'' \cite{ig:diffeology} in the diffeological case.
But if a diffeological groupoid $C$ is a manifold and $l$
and $r$ are submersions, then $C$ is a Lie groupoid in the usual
sense.  

Defining the Lie algebroid of a Lie groupoid is not so simple,
even when $C_0$ is a single point, in which case
$C_1$ is a diffeological group.  A ``Lie algebra'' bracket
for such a group  $G$
 is defined in \cite{he-ma:diffeological} using the conjugation
operation of $G$ on itself.  It is a bilinear
operation on the tangent (vector) space $T_eG$ at the
identity, but antisymmetry and the Jacobi identity have been
established in \cite{le:diffeological} under some extra assumptions on $T_e
G$, holding for example in the case where $G$ is the group of diffeomorphisms of a
manifold $M$ with the functional diffeology.  In this case,
$T_e G$ is, as expected, the space $\calX M$ of vector fields on $M$,
and the bracket is the usual Lie algebra bracket.

If we replace $M$ above by a diffeological space $X$, it is no longer even
clear that inversion is smooth in the group of diffeomorphisms,
though this is the case for many $X$.  We can always define a
stronger diffeology by admitting as plots only those whose composition
with inversion is a plot in the functional diffeology.  The tangent
space at the identity may be identified with those vector fields which
are tangent to paths of diffeomorphisms.  

Let $G$ be a diffeological groupoid, and let $B(G)$ be its group of
bisections, i.e. smooth sections $\gamma$ of $r$ for which $l \circ
\gamma$ is a diffeomorphism.  We give $B(G)$ the diffeology in which
the plots are those plots for the subspace-functional diffeology for
which composition with inversion is also a plot.  The tangent space to
the identity then consists of smooth maps $TG_1 \stackrel{a}{\from} G_0$
lifting the unit section $\epsilon$ which are tangent to
smooth paths through the identity in $B(G)$.   The values of $a$ are
tangent to the $r$-fibres, so it is natural to consider them as
sections of the ``bundle'' $A(G)$ over $G_0$ which is the pullback by
$\epsilon$ of $\ker Tr$.  Without further assumptions, $A(G)$ is not a
vector bundle, but it should still play the role of the Lie algebroid.
To get a bracket operation, we follow \cite{he-ma:diffeological} and
\cite{le:diffeological} and use the natural action 
of the group $B(G)$ on $A(G)$.  Its
derivative at the identity with respect to the first variable 
gives a binary operation $[a,b]$ on sections of $A(G)$, defined {\em
a priori} only when $a$ is {\bf admissible} in the sense of being
tangent to a path in $B(G)$.  
In addition, there is an anchor which takes admissible 
sections of $A(G)$ to admissible vector fields, i.e. admissible
sections of $TG_0$.

Since the action of $B(G)$ preserves all the structures in sight, the 
operation of bracketing on the left by an admissible section is a
derivation with respect to both the bracket itself and multiplication
by functions.  In other words, we have:
$[a,fb]= f[a,b]+ (\rho(a)f)b$ for admissible $a$, functions $f$, and 
all sections $b$, and 
$[a,[b,c]]=[[a,b],c]+[b,[a,c]]$ for admissible $a$, $b$, and $c$.

\end{document}